\documentclass[envcountsame]{llncs}
\usepackage{amsmath,amssymb}
\usepackage{stmaryrd}
\usepackage[hyperindex,breaklinks,colorlinks=true,linkcolor=black,anchorcolor=black,citecolor=black,filecolor=black,menucolor=black,runcolor=black,urlcolor=black,pagebackref]{hyperref}
\usepackage{tikz}
\usepackage[utf8]{inputenc}

	\usepackage{xcolor}
	\usepackage{soul}

\title{Computability and Analysis,\\a Historical Approach}

\author{Vasco Brattka\inst{1,2}\thanks{Vasco Brattka is supported by the National Research Foundation of South Africa.
This article uses some historical insights that were established in \cite{AB14}.}}
             \institute{ Faculty of Computer Science, Universit\"at der Bundeswehr M\"unchen, Germany
             \and Dept.\ of Mathematics \& App.\ Maths., University of Cape Town, South Africa
           \email{Vasco.Brattka@cca-net.de}}

\date{\today}

\def\CC{{\mathcal C}}

\def\IC{{\mathbb{C}}}

\def\IN{{\mathbb{N}}}

\def\IR{{\mathbb{R}}}

\def\TO{\Longrightarrow}
\def\In{\subseteq}

\def\mto{\rightrightarrows}

\def\dom{{\rm dom}}
\def\range{{\rm range}}

\def\Tr{{\rm Tr}}

\def\WKL{\text{\rm\sffamily WKL}}

\def\IVT{\text{\rm\sffamily IVT}}

\def\BWT{\text{\rm\sffamily BWT}}

\def\BFT{\mbox{\rm\sffamily BFT}}

\def\C{\mbox{\rm\sffamily C}}

\def\MCT{\text{\rm\sffamily MCT}}

\def\FRR{\text{\rm\sffamily FRR}}
\def\BIM{\text{\rm\sffamily BIM}}
\def\MAX{\text{\rm\sffamily MAX}}
\def\Z{\text{\rm\sffamily Z}}

\def\leqT{\mathop{\leq_{\mathrm{T}}}}

\def\leqW{\mathop{\leq_{\mathrm{W}}}}
\def\equivW{\mathop{\equiv_{\mathrm{W}}}}

\DeclareMathOperator*{\bigtimes}{\vartimes}

\begin{document}

\maketitle

\begin{abstract}
The history of computability theory and and the history of analysis are surprisingly intertwined since the beginning of the twentieth century.
For one, {\'E}mil Borel discussed his ideas on computable real number functions in his introduction to measure theory. 
On the other hand, Alan Turing had computable real numbers in mind when he introduced his now famous machine model.
Here we want to focus on a particular aspect of computability and analysis, namely on computability properties
of theorems from analysis. This is a topic that emerged already in early work of Turing, Specker and other pioneers of 
computable analysis and eventually leads us to the very recent project of classifying the computational content of theorems
in the Weihrauch lattice. 
\end{abstract}

\section{Introduction}

Probably {\'E}mil Borel was the first mathematician who had an intuitive understanding of computable real number
functions and he anticipated some basic ideas of computable analysis as early as at the beginning of the 20th century. 
It was in his introduction to measure theory where he felt the need to discuss such concepts and we can find
for instance the following crucial observation in \cite{Bor12,Bor50}.

\begin{theorem}[Borel 1912]
Every computable real number function $f:\IR^n\to\IR$ is continuous.
\end{theorem}

Strictly speaking, Borel's definition of a computable real number function was a slight variant of the modern definition 
(see \cite{AB14} for details and translations)
and his definition was informal in the sense that no rigorous notion of computability or of an algorithm was available at
Borel's time. 

It was only Alan Turing who introduced such a notion with the help of his now famous machine model \cite{Tur37,Tur38}.
Interestingly, also Turing was primarily interested in computable real numbers (hence the title of his paper!) 
and not so much in functions and sets on natural numbers that are the main objects of study in modern computability theory. 
Turing's definition of a computable real number function is also a slight variant of the modern definition (see \cite{AB14} for details).

We conclude that computability theory was intertwined with analysis since its early years and here we want to focus on a particular
aspect of this story that is related to computability properties of theorems in analysis, which are one subject of interest
in modern computable analysis \cite{PR89,Ko91,Wei00,BHW08}.

\section{Some Theorems from Real Analysis}

In his early work \cite{Tur37} Turing already implicitly discussed the computational content of some 
classical theorems from analysis. Some of his rather informal observations have been made precise later 
by Specker and others \cite{AB14}. Ernst Specker was probably the first one who actually gave a definition of computable real number functions
that is equivalent to the modern one \cite{Spe49,Spe59}. 
The following theorem is one of those theorems that are implicitly discussed by Turing in \cite{Tur37}.

\begin{theorem}[Monotone Convergence Theorem]
\label{thm:MCT}
Every monotone increasing and bounded sequence of real numbers $(x_n)_n$ 
has a least upper bound $\sup_{n\in\IN}x_n$.
\end{theorem}

What Turing observed (without proof) is that for a computable sequence $(x_n)_n$ of this type
the least upper bound is not necessarily computable. A rigorous proof of this result was presented
only ten years later by Specker \cite{Spe49}.

\begin{proposition}[Turing 1937, Specker 1949]
\label{prop:MCT}
There is a computable monotone increasing and bounded sequence $(x_n)_n$ of real numbers
such that $x=\sup_{n\in\IN}x_n$ is not computable.
\end{proposition}

Specker used (an enumeration of) the halting problem $K\In\IN$ to construct a corresponding sequence $(x_n)_n$
and such sequences are nowadays called {\em Specker sequences}. Then the corresponding non-computable least
upper bound is 
$x=\sum_{i\in K}2^{-i}$.
While the Monotone Convergence Theorem is an example of a non-computable theorem, 
Turing also discusses a (special case) of the Intermediate Value Theorem \cite{Tur37}, which is somewhat better behaved.

\begin{theorem}[Intermediate Value Theorem]
\label{thm:IVT}
Every continuous function\linebreak 
$f:[0,1]\to\IR$ with $f(0)\cdot f(1)<0$ has a zero $x\in[0,1]$.
\end{theorem}

And then Turing's observation, which was stated for the general case by Specker \cite{Spe59}
could be phrased in modern terms as follows.

\begin{proposition}[Turing 1937, Specker 1959]
\label{prop:IVT}
Every computable function $f:[0,1]\to\IR$ with $f(0)\cdot f(1)<0$ has a computable zero $x\in[0,1]$.
\end{proposition}

A rigorous proof could utilize the trisection method, a constructive variant of the well-known bisection method
and can be found in \cite{Wei00}.
Hence, while the Monotone Convergence Theorem does not hold computably, the 
Intermediate Value Theorem does hold computably, at least in a {\em non uniform} sense.
It was claimed by Specker \cite{Spe59} (without proof) and later proved by Pour-El and Richards \cite{PR89}
that this situation changes if one considers a sequential version of the Intermediate Value Theorem.

\begin{proposition}[Specker 1959, Pour-El and Richards 1989]
\label{prop:IVT2}
There exists a computable sequence $(f_n)_n$ of computable functions $f_n:[0,1]\to\IR$ with\linebreak
$f_n(0)\cdot f_n(1)<0$ for all $n\in\IN$
and such that there is no computable sequence $(x_n)_n$ of real numbers $x_n\in[0,1]$ with
$f_n(x_n)=0$.
\end{proposition} 

For their proof Pour-El and Richards used two c.e.\ sets that are computably inseparable.
Their result indicates that the Intermediate Value Theorem does not hold computably in a {\em uniform} sense.
In fact, it is known that the Intermediate Value Theorem is not computable in the following fully uniform sense:
namely, there is no algorithm that given a program for $f:[0,1]\to\IR$ with $f(0)\cdot f(1)<0$ produces
a zero of $f$. Nowadays, we can express this as follows with a partial multi-valued map \cite{Wei00}.

\begin{proposition}[Weihrauch 2000]
$\IVT:\In\CC[0,1]\mto[0,1],f\mapsto f^{-1}\{0\}$
with $\dom(\IVT)=\{f\in\CC[0,1]:f(0)\cdot f(1)<0\}$ is not computable.
\end{proposition}

For general represented spaces $X,Y$ we denote by $\CC(X,Y)$ the space of continuous functions $f:X\to Y$ endowed with a suitable representation \cite{Wei00}
(and the compact open topology) and we use the abbreviation $\CC(X):=\CC(X,\IR)$.
In fact, $\IVT$ is not even continuous and this observation is related to the fact that the Intermediate Value Theorem has no
constructive proof \cite{Bee85}. 

Another theorem discussed by Specker \cite{Spe59} is the Theorem of the Maximum.

\begin{theorem}[Theorem of the Maximum]
\label{thm:MAX}
For every continuous function $f:[0,1]\to\IR$ there exists a point $x\in[0,1]$
such that $f(x)=\max f([0,1])$.
\end{theorem}

Grzegorczyk \cite{Grz55} raised the question whether every computable function\linebreak
$f:[0,1]\to\IR$
attains its maximum at a computable point. This question was answered in the negative
by Lacombe \cite{Lac55d} (without proof) and later independent proofs were provided by Lacombe \cite[Theorem~VI and VII]{Lac57}
and Specker \cite{Spe59}.

\begin{proposition}[Lacombe 1957, Specker 1959]
\label{prop:MAX}
There exists a computable function $f:[0,1]\to\IR$ such that there is no computable $x\in[0,1]$
with $f(x)=\max f([0,1])$.
\end{proposition}

Similar results have also been derived by Zaslavski{\u{i}} \cite{Zas55}.
Specker used a Kleene tree for his construction of a counterexample. A Kleene tree is  a computable counterexample to Weak K\H{o}nig's Lemma.

\begin{theorem}[Weak K\H{o}nig's Lemma]
\label{thm:WKL}
Every infinite binary tree has an infinite path.
\end{theorem}

Kleene \cite{Kle52a} has proved that such counterexamples exist and (like Proposition~\ref{prop:IVT2}) this can be easily achieved
using two computably inseparable c.e.\ sets.

\begin{proposition}[Kleene 1952]
\label{prop:WKL}
There exists a computable infinite binary tree without computable paths.
\end{proposition}

It is interesting to note that even though computable infinite binary trees do not need
to have computable infinite paths, they do at least have paths that are {\em low}, which means
that the halting problem relative to this path is not more difficult than the ordinary halting problem.
In this sense low paths are ``almost computable''. The existence of such solutions has been proved
by Jockusch and Soare in their now famous Low Basis Theorem \cite{JS72}.

\begin{theorem}[Low Basis Theorem of Jockusch and Soare 1972]
\label{thm:LBT}
Every computable infinite binary tree has a low path.
\end{theorem}

Such low solutions also exist in case of the Theorem of the Maximum~\ref{thm:MAX} for computable instances.
The Monotone Convergence Theorem~\ref{thm:MCT} is an example of a theorem where not even 
low solutions exist in general, e.g., Specker's sequence is already
an example of a computable monotone and bounded sequence with a least upper bound
that is equivalent to the halting problem and hence not low. 

Another case similar to the Theorem of the Maximum is the Brouwer Fixed Point Theorem.

\begin{theorem}[Brouwer Fixed Point Theorem]
\label{thm:BFT}
For every continuous function $f:[0,1]^k\to[0,1]^k$ there exists a point $x\in[0,1]^k$ such that $f(x)=x$.
\end{theorem}

It is an ironic coincidence that Brouwer, who was a strong proponent of intuitionistic mathematics
is most famous for his Fixed Point Theorem that does not admit a constructive proof. 
It was Orevkov \cite{Ore63} who proved in the sense of Markov's school that there is a computable counterexample and
Baigger \cite{Bai85} later proved this result in terms of modern computable analysis.

\begin{proposition}[Orevkov 1963, Baigger 1985]
\label{prop:BFT}
There is a computable function $f:[0,1]^2\to[0,1]^2$ without a computable $x\in[0,1]^2$
with $f(x)=x$.
\end{proposition}

One can conclude from Proposition~\ref{prop:IVT} that such a counterexample cannot
exist in dimension $k=1$. Baigger also used two c.e.\ sets that are computably inseparable 
for his construction. 

Yet another theorem with interesting computability properties is the Theorem of Bolzano-Weierstra\ss{}.

\begin{theorem}[Bolzano-Weierstra\ss{}]
\label{thm:BWT}
Every sequence $(x_n)_n$ in the unit cube $[0,1]^k$ has a cluster point.
\end{theorem}

It was Rice \cite{Ric54} who proved that a straightforward computable version of the Bolzano-Weierstra\ss{} Theorem does not hold 
and Kreisel pointed out in his review of this article for the Mathematical Reviews of the American Mathematical Society 
that he already proved a more general result \cite{Kre52}.

\begin{proposition}[Kreisel 1952, Rice 1954]
\label{prop:BWT}
There exists a computable sequence $(x_n)_n$ in $[0,1]$ without a computable cluster point.
\end{proposition}

In fact, this result is not all too surprising and a simple consequence of Proposition~\ref{prop:MCT}.
What is more interesting is that in case of the Bolzano-Weierstra\ss{} Theorem there are even computable bounded sequences
without {\em limit computable} cluster point. Here a point is called {\em limit computable} if it is the limit of a computable sequence.
However, this was established only much later by Le Roux and Ziegler \cite{LZ08a} (answering a question posed by Giovanni Lagnese on the email
list {\em Foundations of Mathematics} [fom] in 2006).

\begin{proposition}[Le Roux and Ziegler 2008]
\label{prop:BWT2}
There exists a computable sequence $(x_n)_n$ in $[0,1]$ without a limit computable cluster point.
\end{proposition}

This is in notable contrast to all aforementioned results that all admit a
limit computable solution for computable instances. For instance, the Monotone Convergence Theorem~\ref{thm:MCT} 
itself implies that every monotone bounded sequence is convergent and hence every computable monotone bounded sequence
automatically has a limit computable supremum. 
Hence, in a certain sense the Bolzano-Weierstra\ss{} Theorem is even less computable than all the 
other results mentioned in this section. 

In this section we have only discussed a selection of theorems that can illustrate a certain variety of possibilities that occur.
The computational content of several other theorems from real analysis has been studied. 
For instance Aberth~\cite{Abe71} constructed a computable counterexample in the Russian sense for 
the Peano Existence Theorem for solutions of ordinary differential equations. Later 
Pour-El and Richards \cite{PR79} constructed another counterexample in the modern sense of computable analysis for this theorem.
The Riemann Mapping Theorem is an interesting example of a theorem from complex analysis that was studied by Hertling \cite{Her99b}.
Now we turn to functional analysis.

\section{Some Theorems from Functional Analysis}

Starting with the work of Metakides, Nerode and Shore \cite{MN82,MNS85} theorems from functional analysis were studied from the perspective
of computability theory. In particular the aforementioned authors studied the Hahn-Banach Theorem. 

\begin{theorem}[Hahn-Banach Theorem]
\label{thm:HBT}
Let $X$ be a normed space over the field $\IR$ with a linear subspace $Y\In X$. 
Then every linear bounded functional $f:Y\to\IR$ has a linear bounded extension $g:X\to\IR$ with $||g||=||f||$.
\end{theorem}

Here $||f||:=\sup_{||x||\leq1}|f(x)|$ denotes the operator norm. The result holds analogously
over the field $\IC$. Here and in the following a {\em computable metric space} $X$ is just a metric space together with a dense sequence
such that the distances can be computed on that sequence (as a double sequence of real numbers).
If the space has additional properties or ingredients, such as a norm that generates the metric, then it is
called  a {\em computable normed space} or in case of completeness also a {\em computable Banach space}. 
If, additionally, the norm is generated by an inner product, then the space is called a {\em a computable Hilbert space}.
A subspace $Y\In X$ is called {\em c.e.\ closed} if there is a computable sequence $(x_n)_n$ in 
$X$ such that one obtains $\overline{\{x_n:n\in\IN\}}=Y$ (where the $\overline{A}$ denotes the closure of $A$).
Metakides, Nerode and Shore \cite{MNS85} constructed a computable counterexample
to the Hahn-Banach Theorem. 

\begin{proposition}[Metakides, Nerode and Shore 1985]
\label{prop:HBT}
There exists a computable Banach space $X$ over $\IR$ with a c.e.\ closed linear subspace $Y\In X$
and a computable linear functional $f:Y\to\IR$ with a computable norm $||f||$ such that every linear 
bounded extension $g:X\to\IR$ with $||g||=||f||$ is non-computable.
\end{proposition}

Similarly, as the computability status of the Brouwer Fixed Point Theorem~\ref{thm:BFT} was dependent on the
dimension $k$ of the underlying space $[0,1]^k$, the computability status of the Hahn-Banach Theorem~\ref{thm:HBT}
is dependent on the dimension and other aspects of the space $X$ \cite{Bra08b}.
Nerode and Metakides \cite{MN82} observed that for finite-dimensional $X$ no counterexample 
as in Proposition~\ref{prop:HBT} exists. However, even in this case the theorem is not uniformly computable \cite{Bra08b}.
Under all conditions that guarantee that the extension is uniquely determined, the Hahn-Banach Theorem is fully computable;
this includes for instance all computable Hilbert spaces \cite{Bra08b}.

A number of further theorems from functional analysis were analyzed by the author of this article and these
include the Open Mapping Theorem, the Closed Graph Theorem and Banach's Inverse Mapping Theorem~\cite{Bra01b,Bra09}.
Another theorem that falls into this category is the Uniform Boundedness Theorem \cite{Bra06}.
These examples are interesting, since they behave differently from all aforementioned examples.
We illustrate the situation using Banach's Inverse Mapping Theorem.

\begin{theorem}[Banach's Inverse Mapping Theorem]
\label{thm:BIM}
If $T:X\to Y$ is a bijective, linear and bounded operator on Banach spaces $X,Y$, then
its inverse $T^{-1}:Y\to X$ is bounded too.
\end{theorem}

Here we obtain the following computable version \cite{Bra09}.

\begin{proposition}[B.\ 2009]
\label{prop:BIM}
If $T:X\to Y$ is a computable, bijective and linear operator on computable Banach spaces $X,Y$, then
its inverse $T^{-1}:Y\to X$ is computable too.
\end{proposition}

That is, every bijective and linear operator $T$ with a program admits also a program for its inverse $T^{-1}$,
but there is not general method to compute such a program for $T^{-1}$ from a program for $T$ in general
as the following result shows \cite{Bra09}.

\begin{proposition}[B.\ 2009]
\label{prop:BIM2}
Inversion ${\mathsf{BIM}}:\In\CC(\ell_2,\ell_2)\to\CC(\ell_2,\ell_2),T\mapsto T^{-1}$ restricted to bijective, linear and bounded $T:\ell_2\to\ell_2$
is not computable (and not even continuous).
\end{proposition}

Analogously, there is also a sequential counterexample \cite{Bra09}.

\begin{proposition}[B.\ 2009]
\label{prop:BIM3}
There exists a computable sequence $(T_n)_n$ of computable, bijective and linear operators $T_n:\ell_2\to\ell_2$
such that the sequence $(T_n^{-1})_n$ of their inverses is not computable.
\end{proposition}

Hence, in several respects the Banach Inverse Mapping Theorem behaves similarly to the Intermediate Value Theorem:
it is non-uniformly computable, but not uniformly computable. Yet we will see that the uniform content of both theorems is different.
 
With our final example of a theorem from functional analysis we want to close the circle
and mention a result that behaves similarly to the Monotone Convergence Theorem~\ref{thm:MCT},
namely the Fr{\'e}chet-Riesz Representation Theorem.

\begin{theorem}[Fr{\'e}chet-Riesz Representation Theorem]
\label{thm:FRT}
For every linear bounded functional $f:H\to\IR$ on a Hilbert space $H$ there exists a unique $y\in H$
such that $f=f_y$ and $||f||=||y||$, where $f_y:H\to\IR,x\mapsto\langle x,y\rangle$. 
\end{theorem}

Here $\langle\;,\;\rangle$ denotes the inner product of the Hilbert space $H$. 
For every computable $y\in H$ the functional $f_y$ is computable with norm $||f_y||=||y||$.
Since the norm $||y||$ is always computable for a computable $y\in H$, it is immediately clear that it suffices to construct a computable functional $f:H\to\IR$ without
a computable norm $||f||$ in order to show that this theorem cannot hold computably. 
In fact, a Specker-like counterexample suffices in this case \cite{BY06}.

\begin{proposition}[B.\ and Yoshikawa 2006]
\label{prop:FRR}
There exists a computable functional $f:\ell_2\to\IR$ such that $||f||$ is not computable and
hence there cannot be a computable $y\in H$ with $f=f_y$ and $||f||=||y||$.
\end{proposition}

\section{A Classification Scheme for Theorems}

If we look at the different examples of theorems that we have presented in the preceding sections then
it becomes clear that theorems can behave quite differently with respect to computability.
For one, the uniform and the non-uniform behavior can differ and the levels of computability
can be of different complexities (computable, low, limit computable, etc.). 
On the other hand, certain theorems seem to be quite similar to each other, for instance
the Monotone Convergence Theorem is similar in its behavior to the Fr{\'e}chet-Riesz Theorem.

This naturally leads us to the question whether there is a classification scheme that allows
to derive all sorts of computability properties of a theorem once it has been classified
according to the corresponding scheme. The best known classification scheme for theorems
in logic is Reverse Mathematics, i.e., the project to classify theorems in second order arithmetic
according to certain axioms that are required to prove the corresponding theorem \cite{Sim09}.
It turns out that this classification scheme is not fine enough for our purposes, because it
only captures theorems in a non-uniform sense. In order to preserve computability properties
such as lowness that are not closed under product, we also need a classification scheme that
is more resource sensitive than reverse mathematics.

Such a classification scheme has been developed over the previous eight years using the concept
of Weihrauch reducibility \cite{GM09,Pau10,Pau10a,BG11,BG11a,BBP12,BGM12,BLP12,BGH15a}.
If $X,Y,Z,W$ are represented spaces, then $f:\In X\mto Y$ is {\em Weihrauch reducible} to $g:\In Z\mto W$, if there
are computable multi-valued functions $H,K$ such that $\emptyset\not=H(x,gK(x))\In f(x)$ for all $x\in\dom(f)$.
In symbols we write $f\leqW g$ in this situation. If the reduction works in both directions,
then we write $f\equivW g$.  It can be shown that this reducibility induces a lattice structure \cite{BG11,Pau10a}.

Now a theorem of logical form $(\forall x\in X)(x\in D\TO (\exists y\in Y)\;P(x,y))$ can  be interpreted
as a multi-valued function $f:\In X\mto Y,x\mapsto \{y\in Y:P(x,y)\}$ with $\dom(f)=D$. For instance,
we obtain the following multi-valued functions for the theorems that we have considered 
in the previous sections (some of which are formulated in greater generality here). 
Here $\Tr$ denotes the set of binary trees and $[T]$ the set of infinite paths of a tree $T$.

\begin{itemize} 
\item $\MCT:\In\IR^\IN\to\IR,(x_n)_n\mapsto\sup_{n\in\IN}x_n$ restricted to monotone bounded sequences.
\item $\IVT:\In\CC[0,1]\mto[0,1],f\mapsto f^{-1}\{0\}$ with $\dom(\IVT):=\{f:f(0)\cdot f(1)<0\}$.
\item $\MAX_X:\In\CC(X)\mto\IR,f\mapsto\{x\in X:f(x)=\max f(X)\}$ for computably compact\footnote{See \cite{BP03} for a definition of computably compact.}
computable metric spaces $X$ and, in particular, for $X=[0,1]$.
\item $\WKL:\In\Tr\mto2^\IN,T\mapsto[T]$ restricted to infinite binary trees.
\item $\BFT_n:\CC([0,1]^n,[0,1]^n)\mto[0,1]^n,f\mapsto\{x:f(x)=x\}$ for $n\geq1$.
\item $\BWT_X:\In X^\IN\mto X,(x_n)_n\mapsto\{x:x$ cluster point of $(x_n)_n\}$, restricted to sequences that are in a compact subset of $X$. 
\item $\BIM_{X,Y}:\In\CC(X,Y)\to\CC(Y,X),T\mapsto T^{-1}$, restricted to bijective, linear, bounded $T$ and for computable Banach spaces $X,Y$.
\item $\FRR_H:\In\CC(H)\to H,f_y\mapsto y$ for computable Hilbert spaces $H$.
\item $\Z_X:\In\CC(X)\to\IR,f\mapsto f^{-1}\{0\}$ for computable metric spaces $X$.
\end{itemize}

We have not formalized the Hahn-Banach Theorem here and point the reader to \cite{GM09}.
The last mentioned problem $\Z_X$ is the {\em zero problem}, which is the problem to find a zero of a continuous function that admits at least one zero.
By \cite[Theorem~3.10]{BP03} we obtain $\Z_X\equivW\C_X$ for the {\em choice problem} of every computable metric space. 
We are not going to define $\C_X$ here, but whenever we use it we will actually take $\Z_X$ as a substitute for it.
The following equivalences were proved in \cite{BG11a}:

\begin{theorem}[Choice on the natural numbers]
\label{thm:CN}
The following are all Weih\-rauch equivalent to each other and complete among functions that are computable with finitely many mind changes \cite{BBP12}:
\begin{enumerate}
\item Choice on natural numbers $\C_{\IN}$.
\item The Baire Category Theorem (in appropriate formulation).
\item Banach's Inverse Mapping Theorem $\BIM_{\ell_2,\ell_2}$.
\item The Open Mapping Theorem for $\ell_2$.
\item The Closed Graph Theorem for $\ell_2$.
\item The Uniform Boundedness Theorem for $\ell_2$.
\end{enumerate}
\end{theorem}

Hence the equivalence class of choice $\C_\IN$ on the natural numbers contains many theorems that are typically
proved with and closely related to the Baire Category Theorem.

We prove that the Theorem of the Maximum is equivalent to the zero problem of $[0,1]$.

\begin{theorem}
\label{thm:MAX-WKL}
$\MAX_X\equivW\Z_X$ for every computably compact computable metric space $X$.
\end{theorem}
\begin{proof}
We prove $\Z_X\leqW\MAX_X$. 
Given a continuous function $f:X\to\IR$ with $A=f^{-1}\{0\}\not=\emptyset$,
we can compute the function $g:X\to\IR$ with $g:=-|f|$. Then $\MAX(g)=f^{-1}\{0\}=A$. 
This proves the claim.
We now prove $\MAX_X\leqW\Z_X$. 
Given a continuous function $f:X\to\IR$ with $\MAX(f)=A\not=\emptyset$, we can
compute $g:X\to\IR$ with $g:=f-\max f(X)$, since $X$ is computably compact. 
Now we obtain $g^{-1}\{0\}=\MAX(f)=A$.
This proves the claim. 
\hfill$\Box$
\end{proof}

We now arrive at the following result that is compiled from different sources. 
It shows that the equivalence class of choice on Cantor space
contains several problems whose non-computably was proved with the help of Weak K\H{o}nig's Lemma or
with the help of two c.e.\ sets that are computably inseparable. 
We point out that the sequential version of the Intermediate Valued Theorem formulated in Proposition~\ref{prop:IVT2}
can be modeled by {\em parallelization}. For $f:\In X\mto Y$ we define its {\em parallelization}  
$\widehat{f}:\In X^\IN\mto Y^\IN,(x_n)_n\mapsto\bigtimes_{n=0}^\infty f(x_n)$, 
which lifts $f$ to sequences. Parallelization is a closure operation in the Weihrauch lattice \cite{BG11}.

\begin{theorem}[Choice on Cantor space]
\label{thm:C2N}
The following are all Weihrauch equivalent to each other and complete among non-deterministically computable functions with a binary sequence as advice \cite{BBP12}:
\begin{enumerate}
\item Choice on Cantor Space $\C_{2^\IN}$.
\item Weak K\H{o}nig's Lemma $\WKL$ \cite{GM09,BG11}.
\item The Theorem of the Maximum $\MAX_{[0,1]}$ (Theorem~\ref{thm:MAX-WKL}).
\item The Hahn-Banach Theorem (Gherardi and Marcone 2009) \cite{GM09}.
\item The parallelization $\widehat{\IVT}$ of the Intermediate Value Theorem \cite{BG11a}.
\item The Brouwer Fixed Point Theorem $\BFT_n$ for dimension $n\geq2$ \cite{BLP12}.
\end{enumerate}
\end{theorem}

We note that \cite{BLP12} contains the proof for the Brouwer Fixed Point Theorem only for dimension $n\geq3$
and the results for $n=2$ is due to Joseph Miller.  
It is easy to see that $\IVT\equivW\BFT_1$ \cite{BLP12}.
We mention that this result implies Propositions~\ref{prop:IVT2}, \ref{prop:MAX}, \ref{prop:WKL}, \ref{prop:BFT} and \ref{prop:HBT}
and constitutes a more general uniform classification. 
In some cases the proofs can easily be derived from known techniques and results, in other cases (for instance for the Hahn-Banach Theorem and the Brouwer Fixed Point
Theorem) completely new techniques are required. 
One can prove that the equivalence classes appearing in Theorems~\ref{thm:CN} and \ref{thm:C2N} are incomparable \cite{BG11a}.
The next equivalence class that we are going to discuss is an upper bound of both.

We first prove that the Monotone Convergence Theorem $\MCT$ is equivalent to the Fr{\'e}chet-Riesz Representation Theorem $\FRR_H$.

\begin{theorem}
\label{thm:FRR-MCT}
$\FRR_{H}\equivW\MCT$ for every computable infinite-dimensional Hilbert space $H$.
\end{theorem}
\begin{proof}
Since every infinite-dimensional computable Hilbert space is computably isometrically isomorphic to $\ell_2$ by \cite[Corollary~3.7]{BY06},
if suffices to consider $H=\ell_2$. 
We first prove $\FRR_{\ell_2}\leqW\MCT$.
Given a functional $f:\ell_2\to\IR$ we need to find a $y\in\ell_2$ such that $f_y=f$ and $||y||=||f||$.
There is a computable sequence $(x_n)_n$ in $\ell_{2}$ such that $\overline{\{x_n:n\in\IN\}}$ is dense in $\{x\in \ell_2:||x||\leq1\}$.
Hence $||f||=\sup_{||x||\leq1}|f(x)|=\sup_{n\in\IN}|f(x_n)|=\sup_{n\in\IN}\max_{i\leq n}|f(x_i)|$ and hence we can compute $||f||$ with the help of $\MCT$.
Now given $f=f_y$ and $||f||$ we can easily compute $y$ be evaluating $f_y(e_n)$ on the unit vectors $e_n$.
We still need to prove $\MCT\leqW\FRR_{\ell_2}$. By \cite[Proposition~9.1]{Bra05} it suffices to show that we can utilize $\FRR_{\ell_2}$ to translate
enumerations $g$ of sets $A\In\IN$ into their characteristic functions. Let us assume that $A=\{n:n+1\in \range(g)\}$.
Without loss of generality we can assume that no value different from zero appears twice in $(g(n))_n$. 
Using the idea of \cite[Example~4.6]{BY06} we choose $a_k:=2^{g(k)-1}$ if $g(k)\not=0$ and $a_k:=0$ otherwise.
Then $a=(a_k)_k\in\ell_2$ and we can compute $f\in\CC(\ell_2)$ with $f(x):=\sum_{k=0}^\infty x_ka_k=\langle x,a\rangle$.
Now, with the help of $\FRR_{\ell_2}$ we obtain a $y\in\ell_2$ with $||y||=||f||=||a||=\sqrt{\sum_{k=0}^\infty|a_k|^2}$.
But using the number $||y||^2$ we can decide $A$, since its binary representation has in the even positions the characteristic function of $A$.
\hfill$\Box$
\end{proof}

We note that $\FRR_H$ for finite-dimensional spaces $H$ is computable. 
Altogether we obtain the following result for this equivalence class. 

\begin{theorem}[The limit]
\label{thm:lim}
The following are all Weihrauch equivalent to each other and complete for limit computable functions:
\begin{enumerate}
\item The limit map $\lim$ on Baire space \cite{Bra05}.
\item The parallelization $\widehat{\C_\IN}$ of choice on the natural numbers \cite{BG11a}.
\item The parallelization $\widehat{\BIM}$ of Banach's Inverse Mapping Theorem \cite{BG11a}.
\item The Monotone Convergence Theorem $\MCT$ \cite{BGM12}.
\item The Fr{\'e}chet-Riesz Representation Theorem $\FRR$ for $\ell_2$ (Theorem~\ref{thm:FRR-MCT}).
\item The Radon-Nikodym Theorem (Hoyrup, Rojas, Weihrauch 2012) \cite{HRW12}.
\end{enumerate}
\end{theorem}

This theorem implies Propositions~\ref{prop:MCT}, \ref{prop:BIM3} and \ref{prop:FRR}. 
 Finally, we mention that we also have a concept of a {\em jump} $f':\In X\mto Y$ for every $f:\In X\mto Y$,
which essentially replaces the input representation of $X$ in such a way that a name of $x\in X$ for $f'$ is a sequence that converges to a name in the sense of $f$.
This makes problems potentially more complicated since less input information is available.
It allows us to phrase results as the following \cite{BGM12}.

\begin{theorem}[B., Gherardi and Marcone 2012]
$\WKL'\equivW\BWT_\IR$.
\end{theorem}

This result does not only imply Proposition~\ref{prop:BWT2}, but also the following \cite{BGM12}.

\begin{corollary}[B., Gherardi and Marcone 2012]
Every computable sequence $(x_n)_n$ in the unit cube $[0,1]^n$ has a cluster
point $x\in[0,1]^n$ that is low relative to the halting problem.
\end{corollary}

Here $x$ is {\em low relative to the halting problem} if $x'\leqT\emptyset ''$ (some authors would only call this a partial relativization of lowness).
In light of Proposition~\ref{prop:BWT2} this is one of the strongest positive properties that one can expect for a cluster point. 
These examples demonstrate that a classification of the Weihrauch degree of a theorem yields a large variety of computability properties of the theorem,
uniform and non-uniform ones on the one hand, and positive and negative ones on the other hand.

\bibliographystyle{splncs03}
\bibliography{C:/Users/Vasco/Dropbox/Bibliography/lit}

\end{document}